\documentclass{birkjour}

\usepackage{amsmath, amsthm, amsfonts, amssymb}

\usepackage{hyperref}

\usepackage{enumerate}
\usepackage{paralist}

\usepackage{mathrsfs}
\usepackage{bm}

\newtheorem{thm}{Theorem}[section]
\newtheorem{cor}[thm]{Corollary}

\newtheorem{prop}[thm]{Proposition}
\theoremstyle{definition}
\newtheorem{dfn}[thm]{Definition}
\theoremstyle{remark}
\newtheorem{rem}[thm]{Remark}
\newtheorem{example}[thm]{Example}

\numberwithin{equation}{section}

\DeclareMathOperator{\mis}{\mathfrak{M}}
\DeclareMathOperator{\Lip}{Lip}

\DeclareMathOperator{\supp}{supp}

\newcommand{\tp}{\otimes}
\newcommand{\htp}{\mathbin{\hat{\otimes}}}
\newcommand{\dm}{\mathbin \diamond}

\newcommand{\e}{\epsilon}

\renewcommand{\c}{\mathscr{C}}
\newcommand{\p}{\mathscr{P}}

\newcommand{\R}{\mathbb{R}}
\newcommand{\C}{\mathbb{C}}
\newcommand{\N}{\mathbb{N}}

\newcommand{\U}{\mathbf{1}}

\newcommand{\A}{{\bm A}}
\newcommand{\fA}{\mathfrak{A}}

\newcommand{\cF}{\mathcal{F}}
\newcommand{\cE}{\mathcal{E}}

\newcommand{\cX}{\mathcal{X}}
\newcommand{\cY}{\mathcal{Y}}

\newcommand{\set}[1]{\{#1\}}
\newcommand{\bigset}[1]{\bigl\{ #1 \bigr\}}

\newcommand{\biggset}[1]{\biggl\{ #1 \biggr\}}

\newcommand{\bigprn}[1]{\bigl( #1 \bigr)}
\newcommand{\Bigprn}[1]{\Bigl( #1 \Bigr)}

\newcommand{\enorm}{\lVert\,\cdot\,\rVert}
\newcommand{\norm}[1]{\lVert #1 \rVert}
\newcommand{\bignorm}[1]{\bigl\lVert #1 \bigr\rVert}
\newcommand{\Bignorm}[1]{\Bigl\lVert #1 \Bigr\rVert}

\newcommand{\tmp}{Tomiyama product}

\begin{document}

\title[Regularity conditions for vector-valued function algebras]%
{Regularity conditions for vector-valued function algebras}

\author{Z. Barqi}
\address{%
School of Mathematics and Computer Sciences,\\
Damghan University, P.O.BOX 36715-364, Damghan, Iran.}
\email{zakiehbarqi@std.du.ac.ir}

\author{M.~Abtahi}
\address{%
School of Mathematics and Computer Sciences,\\
Damghan University, P.O.BOX 36715-364, Damghan, Iran.}
\email{abtahi@du.ac.ir}

\thanks{Corresponding author (M. Abtahi) Tel: +982335220092, Email: abtahi@du.ac.ir}

\subjclass{Primary 46J10; Secondary 46J20.}

\keywords{Vector-valued function algebra, Regularity condition, Ditkin's condition,
Bounded relative unit, Tensor product.}

\begin{abstract}
  We consider several notions of regularity, including strong regularity,
  bounded relative units, and Ditkin's condition, in the setting of vector-valued function
  algebras. Given a commutative Banach algebra $A$ and a compact space $X$, let $\A$ be
  a Banach $A$-valued function algebra on $X$ and let $\fA$ be the subalgebra of
  $\A$ consisting of scalar-valued functions. This paper is about the connection between
  regularity conditions of the algebra $\A$ and the associated algebras $\fA$ and $A$.
  That $\A$ inherits a certain regularity condition $\p$ to $\fA$ and $A$ is the easy part
  of the problem. We investigate the converse and show that, under certain conditions,
  $\A$ receives $\p$ form $\fA$ and $A$. The results apply to tensor products of
  commutative Banach algebras as they are included in the class of vector-valued
  function algebras.
\end{abstract}

\maketitle

\section{Introduction}
\label{sec:intro}

In this paper, we shall investigate problems concerning several regularity conditions
of vector-valued function algebras on compact Hausdorff spaces. Assume that $X$ is
a compact Hausdorff space and that $A$ is a commutative unital Banach algebra over
the complex field $\C$. The connection between certain Banach-algebraic properties
of the algebra $A$ and the associated $A$-valued function algebra
$\c(X, A)$ is investigated in \cite{Nikou-Ofarrell-Ditkin}. These properties concern Ditkin's condition
and bounded relative units. It is proved that $\c(X,A)$ is a (strong) Ditkin algebra if and only if $A$
is one, and that $\c(X,A)$ admits bounded relative units if and only if $A$ does.
Also in \cite{Nikou-Ofarrell-Ditkin}, the relationship between these properties in a special class of
admissible Banach $A$-valued function algebras is considered.

We consider several notions of regularity, including strong regularity, bounded relative
units, and Ditkin's condition, in the general setting of vector-valued function algebras.
To every Banach $A$-valued function algebra $\A$ is associated a complex function algebra $\fA$,
namely, the subalgebra of $\A$ consisting of scalar-valued functions. In the presence of continuous
homomorphisms from $\A$ onto $\fA$ and $A$, many Banach algebraic properties of $\A$ are inherited
by $\fA$ and $A$. The interesting part is whether $\A$ receives a certain property from $\fA$ and $A$.
More precisely, we let $\p$ be a regularity condition that an algebra may satisfy. Our main concern
is to establish an implication as follows;
\begin{equation}\label{P}
     \text{$\A$ has $\p$ if and only if $\fA$ and $A$ have $\p$}.
\end{equation}

We will see that \eqref{P} holds for $\p=$ regularity and $\p=$ bounded relative units.
If $\fA A$, the subalgebra of $\A$ generated by $\fA\cup A$, is dense in $\A$, then \eqref{P}
is true for $\p=$ strong regularity and $\p=$ strong Ditkin's condition. If $\A$ is an $A$-valued
uniform algebra and $\fA$ admits bounded relative units, then \eqref{P} holds for $\p=$ strong regularity
and $\p=$ (strong) Ditkin's condition. The results apply to tensor products of commutative Banach algebras,
as they are included in the class of vector-valued function algebras.

This paper is outlined as follows. Section \ref{sec:pre} contains preliminary materials about
vector-valued function algebras. Also, basic definitions and concepts regarding several notions
of regularity are gathered in Section \ref{sec:pre}. Main results of the paper appear in
Section \ref{sec:main}; it is divided into subsections each of which devoted to establishing
\eqref{P} for a specific regularity condition $\p$.

\section{Preliminaries}
\label{sec:pre}

Let $A$ be a commutative unital Banach algebra over
the complex field $\C$. Every nonzero homomorphism $\phi:A\to\C$ is called
a \emph{character} of $A$. Denoted by $\mis(A)$, the set of all characters
of $A$ is nonempty and its elements are automatically continuous \cite[Lemma 2.1.5]{CBA}.
Consider the Gelfand transform $\hat A=\set{\hat a:a\in A}$, where $\hat a:\mis(A)\to\C$
is defined by $\hat a(\phi)=\phi(a)$, $\phi\in\mis(A)$. The \emph{Gelfand topology} of $\mis(A)$ is
the weakest topology with respect to which every $\hat a\in \hat A$ is continuous.
Endowed with the Gelfand topology, $\mis(A)$ is a compact Hausdorff space,
and the Gelfand transform $A \to \hat A$, $a\mapsto \hat a$, is an algebra homomorphism of $A$
onto the subalgebra $\hat A$ of continuous complex functions on $\mis(A)$. It is well-known
that the Gelfand transform is one-to-one if and only if $A$ is semisimple.
For more on the theory of commutative Banach algebras see, for example,
\cite{BD,Dales,CBA}.

Let $X$ be a compact Hausdorff space, and let $\c(X)$ be the algebra of all
continuous complex functions on $X$, equipped with the uniform norm.
A \emph{Banach function algebra} on $X$ is a subalgebra $\fA$ of $\c(X)$ that
separates the points of $X$, contains the constant functions, and is equipped with some
complete algebra norm $\enorm$. If the norm of $\fA$ is equivalent to the uniform norm
on $X$, then $\fA$ is called a \emph{uniform algebra} on $X$.

If $\fA$ is a Banach function algebra on $X$, for every $x\in X$, the evaluation
homomorphism $\e_x:\fA\to\C$, $f\mapsto f(x)$, is a character of $\fA$, and the mapping
$X\to\mis(\fA)$, $x\mapsto\e_x$, imbeds $X$ onto a compact subset of $\mis(\fA)$.
When this mapping is surjective we call $\fA$ \emph{natural} and $\mis(\fA)=X$;
\cite[Definition 4.1.3]{Dales}. For example, $\c(X)$ is a natural uniform algebra.

Complex function algebras are important and have been extensively investigated by
many authors. Note that every semisimple commutative unital Banach algebra $A$ can be seen,
through its Gelfand transform, as a natural Banach function algebra on $\cX=\mis(A)$.
See, for example, \cite[Chap.~4]{Dales} and \cite[Chap.~2]{CBA}.

\subsection{Vector-valued function algebras}
Throughout the paper, unless otherwise stated, $X$ is a compact Hausdorff space,
and $A$ is a semisimple commutative unital Banach algebra. The space of all continuous
$A$-valued functions on $X$ is denoted by $\c(X,A)$. Algebraic operations are defined in
the obvious way. The uniform norm of a function $f\in\c(X,A)$ is defined by
$\|f\|_X=\sup\set{\|f(x)\|:x\in X}$.

Given an element $a\in A$, the same notation $a$ is used for the constant
function given by $a(x)=a$, for all $x\in X$, and $A$ is regarded
as a closed subalgebra of $\c(X,A)$. We denote the unit element of $A$ by $\U$, and
identify $\C$ with the closed subalgebra $\C\U=\set{\alpha\U:\alpha\in\C}$ of $A$.
Therefore, every function $f\in\c(X)$ can be seen as an $A$-valued function
$x\mapsto f(x)\U$; we use the same notation $f$ for this function,
and regard $\c(X)$ as a closed subalgebra of $\c(X,A)$.

\begin{dfn}[\cite{Abtahi-BJMA,Nikou-Ofarrell-2014}]
  A subalgebra $\A$ of $\c(X,A)$ is called a \emph{Banach $A$-valued function algebra} on $X$
  if
  \begin{enumerate}
    \item $\A$ contains the constant functions $X\to A$, $x\mapsto a$, with $a\in A$,
    \item $\A$ separates the points of $X$ in the sense that, for every pair $x, y$ of distinct
          points in $X$, there exists $f\in \A$ such that $f(x)\neq f(y)$,
    \item $\A$ is equipped with some complete algebra norm $\enorm$.
  \end{enumerate}
  If the norm on $\A$ is equivalent to the uniform norm $\enorm_X$, then $\A$ is called
  an \emph{$A$-valued uniform algebra}.
\end{dfn}

We are mostly interested in those $A$-valued function algebras that are invariant under
composition with characters of $A$. By definition, an $A$-valued function algebra $\A$ is called
\emph{admissible} if $\phi\circ f \in \A$ whenever $f\in \A$ and $\phi\in \mis(A)$.
When $\A$ is admissible, we let $\fA$ be the subalgebra of $\A$ consisting of
scalar-valued functions; that is, $\fA=\A\cap\c(X)$. In this case, $\fA=\set{\phi\circ f : f\in\A}$,
for every $\phi\in\mis(A)$. The algebra $\c(X,A)$ is the most basic example of an admissible $A$-valued
uniform algebra with $\fA=\c(X)$. Also, if $(X, A, \fA, \A)$ is an admissible quadruple, in the sense
of \cite{Nikou-Ofarrell-2014}, then $\A$ is an admissible Banach $A$-valued function algebra
on $X$ and $\fA = \A \cap \c(X)$.

\begin{rem}
  Tensor products of commutative Banach algebras can be regarded as vector-valued function algebras.
  Indeed, suppose that $\fA$ and $A$ are Banach function algebras on their character spaces
  $\cX=\mis(\fA)$ and $\cY=\mis(A)$, respectively. Consider the algebraic tensor product $\fA\otimes A$
  and identify every tensor element $f \tp a$ in $\fA\otimes A$ with the $A$-valued function
  \begin{equation*}
    \hat f\cdot a:\cX\to A, \quad \xi \mapsto \hat f(\xi)a = \xi(f)a.
  \end{equation*}

  If $\gamma$ is an algebra cross-norm on $\fA\otimes A$, not less than the injective tensor product norm,
  then its completion $\fA \hat\tp_\gamma A$ forms a Banach $A$-valued function algebra on $\cX$.
  Similarly, $\fA \hat\tp_\gamma A$ can be regarded as a Banach $\fA$-valued function algebra on $\cY$.
  Adopting the terminology from \cite{Nikou-Ofarrell-Ditkin}, we call $\fA \hat\tp_\gamma A$ a \emph{\tmp}
  of $\fA$ and $A$.
  By a theorem of Tomiyama \cite{Tomiyama},
  $\mis(\fA \hat\tp_\gamma A) = \mis(\fA)\times \mis(A)$. More details can be found in
  \cite[Theorem 6.8]{Abtahi-Farhangi-spectra} and \cite[Proposition 1.6]{Nikou-Ofarrell-Ditkin}.
  For background on cross norms and tensor products of Banach algebras,
  see \cite[Section 2.11]{CBA} and \cite{Ryan}.
\end{rem}

Identifying the character space of a vector-valued function algebra is an interesting problem.
Let $\A$ be an admissible Banach $A$-valued function algebra on $X$, with
$\fA=\A\cap \c(X)$. Take characters $\psi\in\mis(\fA)$ and $\phi\in \mis(A)$ and define
\begin{equation*}
  \psi \diamond \phi :\A\to\C,\quad
  \psi \diamond \phi(f)=\psi(\phi\circ f).
\end{equation*}

Then $\psi \diamond \phi$ is a character of $\A$. In \cite{Abtahi-character-space}, it is shown that,
under certain conditions, the converse is also true; that is, for every $\tau\in\mis(\A)$ there exist
$\psi\in\mis(\fA)$ and $\phi\in \mis(A)$ such that $\tau=\psi \diamond \phi$.
In this case, $\mis(\A)=\mis(\fA)\times \mis(A)$.

\begin{dfn}
  Let $\A$ be an admissible Banach $A$-valued function algebra on $X$.
  It is said that $\A$ is \emph{natural} on $X$ if every character $\tau\in\mis(\A)$ is of the form
  $\tau=\e_x\dm \phi$, for some $x\in X$ and  $\phi\in \mis(A)$.
\end{dfn}

When $\A$ is natural, $\mis(\A)$ is homeomorphic to $X\times \mis(A)$.
Hausner, in \cite{Hausner}, proved that $\{\e_x\dm \phi:x\in X, \phi\in\mis(A)\}$
are the only characters of $\c(X,A)$ so that it is a natural $A$-valued uniform
algebra.

\bigskip

We conclude the section by presenting some examples.

\begin{example}
  Suppose that $X$ is a compact set in the complex plane $\C$.
  Let $P_0(X,A)$ be the algebra of
  the restriction to $X$ of all $A$-valued polynomials of the form
  $p(z) = a_nz^n+\dotsb+a_1z+a_0$, where $n\in\N_0$ and $a_0,a_1,\dotsc,a_n\in A$.
  Also, let $R_0(X,A)$ be the algebra of the restriction to $X$ of all rational
  functions of the form $p(z)/q(z)$, where $p(z)$ and $q(z)$ are $A$-valued polynomials
  and $q(z)$ is invertible in $A$, for $z \in X$. We denote the uniform closures of
  $P_0(X,A)$ and $R_0(X,A)$ by $P(X,A)$ and $R(X,A)$, respectively, and
  write $P(X)$ and $R(X)$ instead of $P(X,\C)$ and $R(X,\C)$. Then
  $P(X,A)$ and $R(X,A)$ are admissible $A$-valued uniform algebras on $X$, with
    \begin{equation*}
       P(X) = P(X,A) \cap \c(X), \quad R(X) = R(X,A) \cap \c(X).
    \end{equation*}
  By \cite[Section 2.5]{CBA}, the character space of
  $P(X)$ is homeomorphic to $\hat X$, the polynomially convex hull of $X$, and
  the character space of $R(X)$ is homeomorphic to $X$.
  A discussion in \cite[Section 2.2]{Abtahi-BJMA}, shows that $P(X,A)$ is isometrically
  isomorphic to the injective tensor product $P(X)\htp_\e A$, so that
  \begin{equation*}
    \mis(P(X,A)) = \mis(P(X))\times \mis(A) = \hat X\times \mis(A).
  \end{equation*}

  Moreover, by \cite[Theorem 2.6]{Abtahi-BJMA}, every function $f\in R_0(X,A)$ can be approximated
  uniformly on $X$ by $A$-valued rational functions of the form $r(z)=r_1(z)a_1+ \dotsb +r_n(z)a_n$,
  where $r_i(z)$, for $1\leq i \leq n$, are complex rational functions in $R(X)$.
  Therefore, the algebra $R(X,A)$ is isometrically isomorphic
  to the injective tensor product $R(X)\htp_\e A$ and
  \begin{equation*}
     \mis(R(X,A)) = \mis(R(X))\times \mis(A) = X\times \mis(A).
  \end{equation*}
\end{example}

\begin{example}\label{exa:Lip(X,A)}
  Let $(X,\rho)$ be a compact metric space. A mapping $f : X \to A$ is called
  an \emph{$A$-valued Lipschitz operator} if
  \begin{equation}
    L(f)=\sup\biggset{\frac{\norm{f(x)-f(y)}}{\rho(x,y)}:x,y\in X,\, x\neq y}<\infty.
  \end{equation}

  The space of $A$-valued Lipschitz operators on $X$ is denoted by $\Lip(X,A)$.
  For any $f\in \Lip(X,A)$, the \emph{Lipschitz norm} of $f$ is defined by
  $\norm{f}_L = \norm{f}_X + L(f)$. In this setting, $\bigprn{\Lip(X,A),\enorm_L}$
  is an admissible Banach $A$-valued function algebra on $X$ with $\Lip(X)=\Lip(X,A)\cap \c(X)$,
  where $\Lip(X)=\Lip(X,\C)$ is the classical complex Lipschitz algebra on $X$.
  The algebra $\Lip(X)$ satisfies all conditions in the Stone-Weierstrass Theorem
  and thus it is dense in $\c(X)$. Using \cite[Lemma 1]{Hausner}, we see that $\Lip(X,A)$
  is dense in $\c(X,A)$. Given $f\in\Lip(X,A)$, if $f(X)$ contains no
  singular element of $A$, then $\U/f\in\Lip(X,A)$. Therefore,
  by \cite[Theorem 2,6]{Abtahi-character-space}, the algebra $\Lip(X,A)$ is natural,
  that is $\mis(\Lip(X,A)) = X\times\mis(A)$.
\end{example}

\subsection{Notions of regularity}
There are many forms of regularity conditions that a Banach function algebra may satisfy.
Many of these conditions have important applications in several areas of functional analysis,
including automatic continuity and the theory of Wedderburn decompositions; see
\cite{Bade-Dales-1992}. Definitions, basic properties and examples concerning several notions
of regularity can be found in \cite[Chap.~4]{CBA}. See also \cite[Sec.~4.1]{Dales}.

A family $\cF$ of complex valued functions on $X$ is called regular if,
given a nonempty closed subset $E$ of $X$ and a point $x\in X\setminus E$, there exists $f\in \cF$
such that $f(x) \neq 0$ and $f|_E = 0$. This leads to the following definition.

\begin{dfn}\label{dfn:regularity}
  The algebra $A$ is called \emph{regular} if its Gelfand transform
  $\hat A$ is regular on $\mis(A)$; that is, given any closed subset $E$ of $\mis(A)$ and
  $\phi_0 \in \mis(A) \setminus E$, there exists $a \in A$ such that $\hat a(\phi_0)\neq 0$
  and $\hat a(\phi) = 0$ for all $\phi \in E$. The algebra $A$ is called \emph{normal}
  if for each proper, closed subset $E$ of $\mis(A)$ and each compact subset $F$
  of $X\setminus E$, there exists $a\in A$ with $\hat a|_F=1$, and $\hat a|_E=0$.
\end{dfn}

\begin{rem}
  Every regular Banach function algebra is normal (\cite[Corollary 4.2.9]{CBA}) and
  natural (Corollary \ref{cor:regular-BFA-are-natural} below).
\end{rem}

To each closed set $E$ in $\mis(A)$
are associated two distinguished ideals;
\begin{align*}
  I_E & = \set{a\in A:\hat a=0 \text{ on } E}, \\
  J_E & = \set{a \in A : \hat a=0 \text{ on a neighbourhood of } E}.
\end{align*}

We abbreviate $I_\phi = I_{\set\phi}$ and $J_\phi = J_{\set\phi}$.
We also write $I_x$ instead of $I_{\e_x}$ and $J_x$ instead of $J_{\e_x}$.

\begin{dfn}\label{dfn:st-regularity}
   The algebra $A$ is called \emph{strongly regular} at $\phi\in\mis(A)$ if $J_\phi$
   is dense in $I_\phi$. It is said that $A$ is strongly regular
   if it is strongly regular at each $\phi\in \mis(A)$.
\end{dfn}

A simple compactness argument shows that if $A$ is strongly regular on a compact set $K$
then it is regular on $K$. In particular, if $A$ is strongly regular then $A$ is regular.
The converse, however, is not true as the following example shows.

\begin{example}
  Let $X=[a,b]$, a compact interval of $\R$.
  For $n\in \N$, let $C^n(X)$ denote the space of $n$-times continuously differentiable
  functions $f:X\to\C$. Norm on $C^n(X)$ is defined by
  \begin{equation*}
    \|f\|_n = \sum_{k=0}^n \frac1{k!} \bignorm{f^{(k)}}_X \quad (f\in C^n(X)).
  \end{equation*}
  By \cite[Theorem 4.4.1]{Dales}, $C^n(X)$ is a natural regular Banach function algebra on $X$
  and, as a discussion after Definition 4.4.2 in \cite{Dales} indicates, it
  is not strongly regular.

\end{example}

The study of strong regularity for uniform algebras was initiated by Wilken \cite{Willken-1969}.
They proved that every strongly regular uniform algebra must be normal and natural.
An elementary proof of Wilken's result is given by Mortini \cite[Proposition 2.4]{Mortini},
which leads to the following stronger result;  see \cite[Theorem 2]{Fein-Somer-1999}.

\begin{thm}[Mortini]\label{thm:if-A-is-st-regular-on-K-and-I(K)=0}
  Let $A$ be a Banach function algebra on $X$, and let
  $K$ be a closed subset of $X$ such that $I_K$ is the zero ideal. If $A$
  is strongly regular on $K$, then $A$ is normal and
  $\mis(A) = X = K$.
\end{thm}

For a survey of the work on strongly regular uniform algebras, see \cite{Fein-Somer-1999}.

\begin{dfn}\label{dfn:BRU}
  Let $A$ be a Banach function algebra on $\mis(A)$. It is said that $A$
  admits \emph{bounded relative units} at $\phi\in\mis(A)$ if there exists a constant $M_\phi > 0$
  such that, for each compact set $K$ in $\mis(A)\setminus\set\phi$, there exists $a\in J_\phi$
  with $\hat a|_K=1$ and $\|a\| \leq M_\phi$.
\end{dfn}

It may happen that $A$ does not have bounded relative units but it is strongly regular.
An example illustrating this situation is presented in \cite{Fein-1999}; see also
\cite[Example 4.1.46]{Dales}.

\begin{dfn}\label{dfn:Ditkin}
  Let $A$ be a Banach function algebra on $\mis(A)$. It is said that $A$ satisfies
  \emph{Ditkin's condition} at $\phi\in \mis(A)$ if $a\in \overline{a J}_\phi$, for all $a\in I_\phi$
  (i.e., given $\e>0$, there exists $b\in J_\phi$ such that $\|a-ab\|<\e$). The algebra $A$ is called
  a \emph{Ditkin algebra}, if it satisfies Ditkin's condition at every point $\phi\in \mis(A)$.
\end{dfn}

Clearly, if $A$ satisfies Ditkin's condition at $\phi\in\mis(A)$ then $A$ is strongly regular at $\phi$.

\begin{dfn}\label{dfn:st-Ditkin}
  Let $A$ be a Banach function algebra on $\mis(A)$. It is said that $A$ satisfies
  \emph{strong Ditkin's condition} at $\phi\in \mis(A)$ if $I_\phi$ has a bounded approximate
  identity contained in $J_\phi$; i.e., there is a bounded net $(u_\alpha)$ in $J_\phi$ such that
  $\|au_\alpha - a\|\to0$, for all $a\in I_\phi$. The algebra $A$ is called a
  \emph{strong Ditkin algebra} if it satisfies strong Ditkin's condition at every $\phi\in \mis(A)$.
\end{dfn}

By \cite[Proposition 3]{Fein-1995-note}, if $A$ is strongly regular at $\phi$ and
$I_\phi$ has a bounded approximate identity, then $A$ satisfies strong Ditkin condition
at $\phi$. 
In \cite{Bade-Open-problem}, Bade asked whether a strong Ditkin algebra admits
bounded relative units. This question is resolved positively in \cite[Theorem 4]{Fein-1995-note},
where it is shown that if a normal, unital Banach function algebra $A$ is strongly regular at
$\phi$, and $I_\phi$ has a bounded approximate identity, then $A$ has bounded
relative units at $\phi$ (see also \cite[Corollary 4.1.33]{Dales}).
Conversely, a Banach function algebra which is both strongly regular and admits bounded relative
units is a strong Ditkin algebra. Indeed, choose $a\in I_\phi$ and let $\e > 0$. Then, there exists
$b\in J_\phi$ such that $\|a-b\|<\e$, and there exists $u\in J_\phi$ such that
$\|u\| < M_\phi$ and $\hat u|_K=1$, where $K=\supp \hat b$,
the closed support of $\hat b$. It follows that $\hat b - \hat a \hat u = (\hat b - \hat a) \hat u$.
Since $A$ is semisimple, we get $b-au = (b-a)u$.
Hence,
\[
  \|a-au\| = \|a-b + (b-a)u\| < \e(1+M_\phi).
\]
By \cite[Corollary 2.9.15]{Dales}, $I_\phi$ has a bounded approximate identity.
Therefore, the class of strong Ditkin algebras is the intersection of
the classes of strongly regular algebras and of algebras with bounded relative units.

\medskip

The above discussion is summarized in the following statement.

\begin{prop}
  Let $A$ be a normal Banach function algebra on $\mis(A)$.
  Then, for every $\phi\in \mis(A)$, the following are equivalent;
  \begin{enumerate}[\upshape(a)]
    \item \label{item:st-Ditkin}
      $A$ satisfies strong Ditkin's condition at $\phi$,

    \item \label{item:has-bai}
      $A$ is strongly regular at $\phi$ and $I_\phi$ has a bounded approximate identity,

    \item \label{item:admits-bru}
      $A$ is strongly regular at $\phi$ and admits bounded relative units at that point.
  \end{enumerate}
\end{prop}

The story for uniform algebras is a little different \cite{Fein-Somer-1999}. If $\fA$ is a uniform algebra,
then $\fA$ has bounded relative units at $\phi$ if, and only if, $\fA$ satisfies strong Ditkin's condition
at $\phi$. In \cite{Fein-1992} an example of a non-trivial, normal uniform algebra $\fA$ is given which
admits bounded relative units.

To complete the discussion, we give the following result. Although proofs may be found in the literature,
we present one for the reader's convenience.

\begin{thm}\label{thm:if-T:A-to-B-is-a-homo}
  Suppose that $A$ and $B$ are commutative unital Banach algebras and that
  $T: A\to B$ is a continuous homomorphism with $\overline{T(A)}=B$.
  Then
  \begin{enumerate}[\upshape\quad(a)]
    \item \label{item:regular}
       if $A$ is (strongly) regular, so is $B$,
    \item \label{item:bru}
       if $A$ admits bounded relative units, so does $B$,
    \item \label{item:St.Ditkin}
       if $A$ satisfies strong Ditkin's condition, so does $B$,
    \item \label{item:Ditkin}
       if $T(A)=B$ and $A$ satisfies Ditkin's condition, so does $B$.
  \end{enumerate}
\end{thm}

\begin{proof}
  Before going to detailed proofs for each statement, notice that the condition $\overline{T(A)}=B$
  implies that $\psi\circ T\in \mis(A)$, for every $\psi\in\mis(B)$. The homomorphism $T:A\to B$,
  therefore, induces a mapping $\hat T:\mis(B)\to\mis(A)$, defined by $\hat T(\psi) = \psi \circ T$.
  Indeed, $\hat T(\psi)(a)= \widehat{T(a)}(\psi)$, for all $a\in A, \psi\in\mis(B)$. It is easily seen that
  $\hat T$ is a continuous injection, so that $\mis(B)$ is homeomorphic to a compact subset of $\mis(A)$,
  and the following statements hold;
  \begin{enumerate}[(1)]
    \item \label{item:E-is-closed-iff}
          a set $E$ in $\mis(B)$ is closed if, and only if, its image
          $\hat T(E)$ in $\mis(A)$ is closed,

    \item \label{item:T(a)-in-I(E)-iff}
          $T(a)\in I_E$ if and only if $a \in I_{\hat T(E)}$,
          where $a\in A$ and $E\subset \mis(B)$,

    \item \label{item:T(a)-in-J(psi)-iff}
          $T(a) \in J_\psi$ if and only if $a \in J_{\hat T(\psi)}$, where
          $a\in A$ and $\psi\in\mis(B)$.
  \end{enumerate}

  \medskip\noindent
  We now present a detailed discussion of theorem statements.

  \eqref{item:regular} It is clear from \eqref{item:E-is-closed-iff}--\eqref{item:T(a)-in-J(psi)-iff}
  that if $A$ is regular then so is $B$.
  Assume that $A$ is strongly regular. Take $\psi_0\in\mis(B)$ and $b_0\in I_{\psi_0}$.
  Let $L > 1 + 2\|T\|$. Given $\e>0$, since $\overline{T(A)}=B$, there exists $a_1\in A$
  such that $\|b_0-T(a_1)\|<\e/L$. Let $\phi_0=\hat T(\psi_0)$ and
  $a_0=a_1-\phi_0(a_1)$. Then $a_0 \in I_{\phi_0}$ and
  \begin{align*}
    \|a_1-a_0\|
      & = |\phi_0(a_1)| = |\psi_0(T(a_1))| \\
      & = |\psi_0(T(a_1) -b_0)| \leq \|T(a_1)-b_0\| < \e / L.
  \end{align*}
  Therefore,
  \begin{equation} \label{eqn:b-0-approximation}
  \begin{split}
    \|b_0-T(a_0)\|
        & \leq \|b_0-T(a_1)\| + \|T(a_1-a_0)\| \\
        & \leq \e/L + \|T\|\|a_1-a_0\|
        \leq \e (1 + \|T\|)/L.
  \end{split}
  \end{equation}

  Since $A$ is strongly regular at $\phi_0$ and $a_0\in I_{\phi_0}$,
  there is $a\in J_{\phi_0}$ such that $\|a_0-a\|<\e/L$.
  Take $b=T(a)$. By \eqref{item:T(a)-in-J(psi)-iff}, $b\in J_{\psi_0}$ and
  \begin{align*}
    \|b_0-b\|
       & = \|b_0-T(a)\|  \\
       & \leq \|b_0-T(a_0)\| + \|T(a_0-a)\| \\
       & \leq \|b_0-T(a_0)\| + \|T\|\|a_0-a\| \\
       & \leq \e(1+\|T\|)/L + \e\|T\|/L \\
       & = \e(1 + 2\|T\|)/L < \e.
  \end{align*}
  Therefore, $B$ is strongly regular at $\psi_0$.

  \eqref{item:bru} It follows from \eqref{item:E-is-closed-iff}--\eqref{item:T(a)-in-J(psi)-iff}.

  \eqref{item:St.Ditkin}
  Let $\psi_0\in\mis(B)$ and $\phi_0=\hat T(\psi_0)$. Since $A$ is a strong Ditkin algebra,
  $I_{\phi_0}$ has a bounded approximate identity $(u_\alpha)$, say, contained in $J_{\phi_0}$.
  Let $v_\alpha=T(u_\alpha)$, for all $\alpha$. Then $(v_\alpha)$ is a bounded net, contained
  in $J_{\psi_0}$ by (3). Given $b_0\in I_{\psi_0}$, we show that $b_0v_\alpha\to b_0$.
  For $\e>0$, as in \eqref{eqn:b-0-approximation}, there is $a_0\in I_{\phi_0}$ such that
  $\|b_0-T(a_0)\|<\e$. Hence,
  \begin{multline*}
    \|b_0-b_0v_\alpha\|
      \leq \|b_0-T(a_0)\|+\|T(a_0 - a_0 u_\alpha)\| \\
       + \|\bigl(T(a_0)-b_0\bigr)T(u_\alpha)\|
      < \e + \|T\| \|a_0-a_0u_\alpha\| + \e M,
  \end{multline*}
  where $M=\sup_\alpha \|T(u_\alpha)\|$. Since $\|a_0-a_0u_\alpha\| \to 0$, we get
  \begin{equation*}
    \limsup_\alpha \|b_0-b_0v_\alpha\| \leq \e(1+M).
  \end{equation*}
  This holds for arbitrary $\e>0$, so $\lim_\alpha \|b_0-b_0v_\alpha\| =0$.

  \eqref{item:Ditkin}
  For $\psi_0\in\mis(B)$, let $b_0\in I_{\psi_0}$ and $\phi_0=\hat T(\psi_0)$.
  Since $T(A)=B$, there is $a_0\in I_{\phi_0}$ such that $T(a_0)=b_0$.
  The algebra $A$ satisfies Ditkin's condition at $\phi_0$, so there exists, for every $\e>0$,
  an element $u_0\in J_{\phi_0}$ such that $\|a_0-a_0u_0\|<\e/\|T\|$.
  Take $v_0=T(u_0)$. Then $v_0\in J_{\psi_0}$ and
    \begin{equation*}
      \|b_0-b_0v_0\|= \|T(a_0)-T(a_0)T(u_0)\| \leq \|T\| \|a_0-a_0u_0\| < \e.
    \end{equation*}
  This means that $B$ satisfies Ditkin's condition at $\psi_0$.
\end{proof}

We conclude the section with a result on vector-valued function algebras.

\begin{thm}\label{thm:if-AA-has-P-then-fA-and-A-have-P}
  Let $\A$ be an admissible Banach $A$-valued function algebra on $X$,
  with $\fA=\c(X)\cap \A$.
  \begin{enumerate}[\upshape(a)]
    \item \label{item:(st)reg}
      If $\A$ is (strongly) regular, so are $\fA$ and $A$.
    \item \label{item:BRU}
      If $\A$ admits bounded relative units, so do $\fA$ and $A$.
    \item \label{item:(st)Ditkin}
      If $\A$ satisfies (strong) Ditkin's condition, so do $\fA$ and $A$.
  \end{enumerate}
\end{thm}

\begin{proof}
  Take a point $x\in X$ and a character $\phi\in \mis(A)$. Consider the following mappings;
  \begin{equation*}
    \left\{\!\!
      \begin{array}{l}
        \cE_{x}:\A\to A, \\
        f\mapsto f(x),
      \end{array}
    \right.
   \qquad
   \left\{\!\!
     \begin{array}{ll}
        \Phi:\A\to \fA, \\
        f \mapsto \phi\circ f.
     \end{array}
   \right.
  \end{equation*}
  Then $\cE_{x}$ and $\Phi$ are continuous epimorphisms.
  Now, apply Theorem \ref{thm:if-T:A-to-B-is-a-homo}.
\end{proof}

\section{Regularity conditions for vector-valued function algebras}
\label{sec:main}

Throughout this section, we let $\A$ be an admissible Banach $A$-valued function
algebra on $X$ with $\mis(\A)=\mis(\fA) \times \mis(A)$, where $\fA=\c(X)\cap \A$.
By Theorem \ref{thm:if-AA-has-P-then-fA-and-A-have-P}, if $\A$ satisfies a regularity condition
$\p$ then $\fA$ and $A$ satisfy $\p$. We are concerned with the question that whether $\A$ receives
$\p$ from $\fA$ and $A$.

\subsection{Regularity}
Let us commence with the following observation.


\begin{prop}\label{prop:if-A-is-reg-then-mis(A)=Xxmis(A)}
  If $\A$ is regular then $\A$ is natural, i.e., $\mis(\A) = X \times \mis(A)$.
\end{prop}

\begin{proof}
 Suppose, towards a contradiction, that $\mis(\A)\neq X\times\mis(A)$, and
 take a character $\tau_0$ in $\mis(\A)\setminus X\times\mis(A)$. Since $\A$ is regular
 and $X\times\mis(A)$ is closed in $\mis(\A)$, there exists a function $f \in \A$
 such that $\tau_0(f)\neq 0$ and $\hat f= 0$ on $X\times\mis(A)$. This means that
 $\phi(f(x))=0$, for all $x\in X$ and $\phi\in \mis(A)$. Since $A$ is assumed to be
 semisimple, we get $f(x)=0$, for all $x\in X$. This means that $f=0$, which is a contradiction.
 Therefore, $\mis(\A)= X\times\mis(A)$.
\end{proof}

In case $A=\mathbb{C}$, we get the following result for complex function algebras.

\begin{cor}\label{cor:regular-BFA-are-natural}
  Let $\fA$ be a Banach function algebra on $X$. If\/ $\fA$ is regular,
  then $\fA$ is natural; that is $\mis(\fA)=X$.
\end{cor}

\begin{thm}\label{thm:regularity-for-vv-FA}
  For the $A$-valued function algebra $\A$, suppose that $\fA$ and $A$ are regular.
  Then $\A$ is regular.
\end{thm}

\begin{proof}
  Let $E$ be a closed subset of $\mis(\A)$ and let $\tau\in \mis(\A) \setminus E$.
  Take $\psi\in \mis(\fA)$ and $\phi\in \mis(A)$ such that $\tau=\psi\dm \phi$.
  Let $W$ be a neighbourhood of $\tau$ in $\mis(\A)$ such that $W\cap E=\emptyset$.
  There exist neighbourhoods $U$ of $\psi$ in $\mis(\fA)$ and $V$ of $\phi$ in $\mis(A)$
  such that $U\times V \subset W$. Since $\fA$ is regular, there exists an element $g\in \fA$
  such that $\hat g=0$ on $\mis(\fA)\setminus U$ and $\hat g(\psi)=1$. Similarly, there exists
  an element $a\in A$ such that $\hat a=0$ on $\mis(A)\setminus V$ and $\hat a(\phi)=1$.
  Now let $f=g a$. Then $f \in \A$, $\hat f(\tau)=1$ and $\hat f=0$ on $F$, where
  \[
    F=\bigl((\mis(\fA)\setminus U)\times \mis(A)\bigr) \cup \bigl(\mis(\fA)\times(\mis(A)\setminus V)\bigr).
  \]

   Since $E\subset F$, we get $\hat f=0$ on $E$. Therefore, $\A$ is regular.
\end{proof}

As a consequence of the above theorem, we get the following result,
which appears in \cite[Theorem 4.2.20]{CBA}.

\begin{cor}\label{cor:regularity-for-fA-tp-A}
  If\/ $\fA$ and $A$ are regular, then any Tomiyama product $\fA\htp_\gamma A$ is regular.
\end{cor}

\subsection{Bounded relative units}
Our first result extends \cite[Theorem 2]{Nikou-Ofarrell-Ditkin}.

\begin{thm}\label{thm:BRU}
  For the $A$-valued function algebra $\A$, suppose that
  $\fA$ and $A$ admit bounded relative units. Then $\A$ admits bounded relative units.
\end{thm}

\begin{proof}
  Let $\tau\in \mis(\A)$.
  Take $\psi\in \mis(\fA)$ and $\phi\in \mis(A)$ such that $\tau=\psi\dm \phi$.
  Let $K$ be a compact set in $\mis(\A)$ such that $\tau\notin K$.
    Let $W$ be a neighbourhood of $\tau$ in $\mis(\A)$ such that $W\cap K=\emptyset$.
  There exist neighbourhoods $U$ of $\psi$ in $\mis(\fA)$ and $V$ of $\phi$ in $\mis(A)$
  such that $U\times V \subset W$. Then $(U \times V) \cap K = \emptyset$. Set $E=\mis(\fA)\setminus U$
  and $F = \mis(A)\setminus V$. Then $E$ and $F$ are compact sets in $\mis(\fA)$ and $\mis(A)$,
  respectively, such that $\psi\notin E$, $\phi\notin F$ and
   \[
     K\subset(E\times\mis(A)) \cup (\mis(\fA)\times F).
   \]

   Since $\fA$ admits bounded relative units at $\psi$, there exists a constant $M_\psi$,
   a function $g\in \fA$, and a neighbourhood $U_1$ of $\psi$, such that $\hat g=0$ on $U_1$,
   $\hat g=1$ on $E$, and $\|g\|\leq M_\psi$. Since $A$ admits bounded relative units at $\phi$,
   there exists a constant $M_\phi$, an element $a\in A$, and a neighbourhood $V_1$ of $\phi$,
   such that $\hat a=0$ on $V_1$, $\hat a=1$ on $F$, and $\|a\|\leq M_\phi$.
   Take $W_1 = U_1\times V_1$ and $f = g+a-ga$. Then $f\in\A$, $\hat f =0$ on $W_1$,
   $\hat f=1$ on $K$, and
     \begin{equation*}
       \|f\|=\| g+a-ga \| \leq M_\psi + M_\phi + M_\psi M_\phi.
     \end{equation*}
   Therefore, $\A$ admits bounded relative units at $\tau$.
\end{proof}

\begin{cor}
  If\/ $\fA$ and $A$ have bounded relative units, then any Tomiyama product
  $\fA \htp_\gamma A$ has bounded relative units.
\end{cor}

The uniform algebra $\c(X)$ has bounded relative units. Therefore,
the algebra $\c(X,A)$ has bounded relative units if, and only if, $A$ has bounded relative units
(\cite[Corollary 1.1]{Nikou-Ofarrell-Ditkin}).

\subsection{Strong regularity}
If $\A$ is strongly regular then,
by Theorem \ref{thm:if-AA-has-P-then-fA-and-A-have-P}~(\ref{item:(st)reg}),
$\fA$ and $A$ are strongly regular. We show that the converse is true if each of the following conditions holds;
\begin{enumerate}
  \item $\A$ is an $A$-valued uniform algebra and $\fA$ admits bounded relative units,
  \item $\fA A$, the subalgebra of $\A$ generated by $\fA \cup A$, is dense in $\A$.
\end{enumerate}

First, we remark that if $\A$ is strongly regular at every point of $X\times\mis (A)$,
then $\A$ is regular and $\mis(\A)=X\times \mis(A)$. To see this, consider $\A$ as a
(complex) Banach function algebra on $\cX=\mis(\A)$ and let $K=X\times \mis(A)$.
Then $K$ is a closed subset of $\cX$ and $I_K=\{0\}$. Now, apply
Theorem \ref{thm:if-A-is-st-regular-on-K-and-I(K)=0} to get $\mis(\A)=K$.

\begin{thm}\label{thm:st-regularity-when-fA-has-BRU}
  Suppose that $\A$ is an $A$-valued uniform algebra and that $\fA$ admits bounded relative units.
  Then $\A$ is strongly regular if and only if $A$ is strongly regular.
\end{thm}

\begin{proof}
  If $\A$ is strongly regular, then $A$ is strongly regular
  by Theorem \ref{thm:if-AA-has-P-then-fA-and-A-have-P}~(\ref{item:(st)reg}).
  Conversely, assume that $A$ is strongly regular. Note that $\fA$ is regular and
  $\mis(\fA)=X$. Given $\tau_0\in \mis(\A)$, there exist $x_0\in X$ and $\phi_0\in \mis (A)$
  such that $\tau_0 = \e_{x_0}\dm\phi_0$.
  Let $f\in I_{\tau_0}$ and $\epsilon > 0$. Let $a=f(x_0)$. Then $\phi_0(a)=0$ and,
  since $A$ is strongly regular at $\phi_0$, there exist a neighbourhood $V_0$ of $\phi_0$
  and an element $b\in A$ such that $\hat b=0$ on $V_0$ and $\| a - b\| < \e$. Set
  \begin{equation*}
     U=\bigset{x\in X : \|f(x)- b\|< \e}.
  \end{equation*}

  Since $\fA$ admits bounded relative units at $x_0$, there exist a constant $M_{x_0}$,
  a function $g\in \fA$ and a neighbourhood $U_0$ of
  $x_0$, such that $g|_{X\setminus U} = 1$, $g|_{U_0}=0$ and $\| g\| < M_{x_0}$.
  Let $G= f - (\U - g)(f - b)$ and $W_0 = U_0\times V_0$.
  Then, for each $\tau=\e_x\dm\phi$ in $U_0\times V_0$, we have
  \[
   \tau(G) = \phi(f(x)) - (\U - g(x))\phi(f(x) - b)
   = \phi(f(x)) - \phi(f(x))=0.
  \]
  Therefore, $G\in J_{\tau_0}$. Moreover, if $x\in U$ then
  \begin{align*}
    \|f(x)-G(x)\|
     & = \| (\U - g(x))(f(x) - b)\|  \\
     & \leq (1+M_{x_0})\|f(x)-b\| \leq (1+M_{x_0})\e,
  \end{align*}
  and if $x\in X\setminus U$ then $\|f(x)-G(x)\| = \| (\U - g(x))(f(x) - b)\| =0$.
  Hence,
    \begin{equation*}
      \|f - G\|_X \leq (1+M_{x_0})\epsilon.
    \end{equation*}

    We conclude that $\A$ is
  strongly regular at $\tau_0$.
\end{proof}

\begin{cor}
  Let $\fA$ be a uniform algebra on $X$ having bounded relative units.
  The injective tensor product $\fA\htp_\e A$ is strongly regular if and only if
  $A$ is strongly regular.
\end{cor}

\begin{cor}
  $\c(X,A)$ is strongly regular if and only if $A$ is strongly regular.
\end{cor}

In the following, $\fA A$ represents the subalgebra of $\A$ generated by $\fA\cup A$.
Therefore, $f\in \fA A$ if and only if $f=\sum_{i=1}^n f_i a_i$, with $n\in\N$,
$f_i\in \fA$ and $a_i\in A$.

\begin{thm}\label{thm:strong-reg-when-fA.A-is-dense}
  Suppose that $\fA A$ is dense in $\A$. If $\fA$ and $A$ are strongly regular,
  then $\A$ is strongly regular.
\end{thm}

\begin{proof}
  Note that the algebra $\A$ is regular and $\mis(\A)= X\times\mis(A)$.
  Take a character $\tau_0=\varepsilon_{x_0}\dm\phi_0$ in $X\times \mis(A)$, let
  $F\in I_{\tau_0}$ and let $\e>0$.

  First, assume that $F\in \fA A$ and it is of the form $F=f_1a_1+\dotsb+f_na_n$,
  with $f_i\in \fA$ and $a_i\in A$, $1\leq i \leq n$, $n\in\N$. Write
  \begin{equation*}
    F = \sum_{i=1}^n \bigl(f_i-f_i(x_0)\bigr)a_i + \sum_{i=1}^n f_i(x_0)a_i
      = \sum_{i=1}^n g_i a_i + a,
  \end{equation*}
  where $g_i=f_i-f_i(x_0)$ and $a = \sum_{i=1}^n f_i(x_0)a_i$. We see that
  \begin{equation*}
    \phi_0(a) = \phi_0\Bigprn{\sum_{i=1}^n f_i(x_0)a_i} = \sum_{i=1}^n f_i(x_0)\phi_0(a_i)
    = \tau_0(F) = 0.
  \end{equation*}

  Since $A$ is strongly regular at $\phi_0$, there exists a neighbourhood $V_0$ of $\phi_0$ in $\mis(A)$,
  and an element $b\in A$ such that $\hat b=0$ on $V_0$ and $\|a - b\| <\epsilon/2$.
  Since $\fA$ is strongly regular at $x_0$ and $g_i(x_0)=0$, for $1\leq i \leq n$,
  there exists, for each $i$, a neighbourhood $U_i$ of $x_0$ in $X$ and a function $h_i\in\fA$,
  such that $h_i=0$ on $U_i$ and $\|h_i - g_i\| < \e/M$, where $M=2\sum_{i=1}^n \|a\|_i$.
  Take
  \begin{equation*}
    G=\sum_{i=1}^n h_i a_i + b \quad \text{and} \quad
     W = \bigcap_{i=1}^n U_i \times V_0.
  \end{equation*}

  Then $W$ is a neighbourhood of $\tau_0$ in $\mis(\A)$ and $\tau=\e_x \dm \phi \in W$
  implies that $x\in U_i$, $h_i(x)=0$, for $1\leq i \leq n$, and $\phi\in V_0$, $\phi(b)=0$.
  Therefore,
  \begin{equation*}
    \tau(G) = \sum_{i=1}^n h_i(x) \phi(a_i) + \phi(b) = 0,
  \end{equation*}
  meaning that $G\in J_{\tau_0}$. Moreover,
  \begin{equation*}
    \|F - G\| \leq \sum_{i=1}^n \|g_i-h_i\|\|a_i\| + \|a-b\| < \frac\e2+\frac\e2 = \e.
  \end{equation*}

  Now, consider the general case of $F\in I_{\tau_0}$. Since $\fA A$ is dense in $\A$,
  there exists $F_1\in \fA A$ such that $\|F-F_1\|<\e/3$. Let $F_0=F_1-\tau_0(F_1)\U$.
  Then $F_0\in \fA A$ and $\tau_0(F_0)=0$. Previous case applies to $F_0$ so that,
  for some function $G \in J_{\tau_0}$, we get $\|F_0 - G\| < \e/3$. We claim that
  $\|F-G\|<\e$. Write
  \begin{equation*}
    F - G = (F-F_0) + (F_0- G)= (F-F_1+\tau_0(F_1)\U)+(F_0-G).
  \end{equation*}

  \noindent
  Note that $\tau_0(F_1) = \tau_0(F_1-F)$ so that $|\tau_0(F_1)| \leq \|F_1-F\| < \e/3$.
  Therefore,
    \begin{equation*}
     \|F - G\| \leq \|F-F_1\| + |\tau_0(F_1)| + \|F_0- G\|
       < \e/3+\e/3+\e/3=\e.
    \end{equation*}
  We conclude that $\A$ is strongly regular at $\tau_0$.
\end{proof}

\begin{cor}
  If the algebras $\fA$ and $A$ are strongly regular, then any Tomiyama product
  $\fA \hat \tp_\gamma A$ is strongly regular.
\end{cor}

\subsection{Ditkin's condition}

If $\A$ is a Ditkin algebra then, by
Theorem \ref{thm:if-AA-has-P-then-fA-and-A-have-P}~(\ref{item:(st)Ditkin}),
$\fA$ and $A$ are Ditkin algebras.
We show that the converse is true provided $\A$ is an $A$-valued uniform algebra and
$\fA$ admits bounded relative units. This extends \cite[Theorem 1]{Nikou-Ofarrell-Ditkin}.

\begin{thm}\label{thm:Ditkin-when-fA-has-BRU}
  Let $\A$ be an $A$-valued uniform algebra such that $\fA$ admits bounded relative units.
  Then $\A$ is a Ditkin algebra if, and only if, $A$ is a Ditkin algebra.
\end{thm}

\begin{proof}
   We just need to consider the ``if'' part of the theorem. Suppose that $A$ is a Ditkin algebra.
   Since the algebra $\fA$ admits bounded relative units, it is regular from which we get
   $\mis(\fA)=X$ and $\mis(\A)=X\times \mis(A)$. Take $x_0\in X$ and $\phi_0\in\mis (A)$,
   and let $\tau_0=\e_{x_0}\dm\phi_0$. We need to show that $f\in \overline{f J}_{\tau_0}$,
   for every $f\in I_{\tau_0}$.

   Given $f\in I_{\tau_0}$ and $\e>0$, let $a=f(x_0)$. Then $\phi_0(a) = \phi_0(f(x_0))
   = \tau_0(f)=0$. Since $A$ satisfies Ditkin's condition at $\phi_0$, we have
   $a\in \overline{aJ}_{\phi_0}$,
   i.e., there exist $b\in A$ and a neighbourhood $V_0$ of $\phi_0$ such that
   $\hat b = 0$ on $V_0$ and $\|a-ab\|<\e$. Set
   \begin{equation*}
     U=\bigset{ x\in X: \| f(x)-f(x)b\|<\e }.
   \end{equation*}
   Then $U$ is a neighbourhood of $x_0$ and $K=X\setminus U$ is compact in $X$.
   Since $\fA$ admits bounded relative units
   at $x_0$, there exists a constant $M_{x_0}$, a function $g\in\fA$, and a neighbourhood
   $U_0$ of $x_0$ such that $g=0$ on $U_0$, $g=1$ on $X\setminus U$, and $\|g\| < M_{x_0}$.
   Let $G=g+b-gb$ and $W_0 = U_0 \times V_0$. Then $\hat{G}=0$ on $W_0$ so that $G\in J_{\tau_0}$.
   If $x\in X\setminus U$ then $G(x)=\U$ and thus $\|f(x)-f(x)G(x)\|=0$. If $x\in U$ then
   \begin{align*}
    \|f(x)-f(x)G(x)\| = \|(\U-g(x))(f(x)-f(x)b)\| < (1+M_{x_0})\e.
   \end{align*}
   Therefore, $\|f-fG\|_X \leq \e(1+M_{x_0})$ and we conclude that $f\in \overline{f J}_{\tau_0}$.
\end{proof}

\begin{cor}\label{cor:fA-tp-A-is-Ditkin-iff-A-is-Ditkin}
  Suppose that $\fA$ is a uniform algebra on $X$ having bounded relative units.
  The injective tensor product $\fA\htp_\e A$ is a Ditkin algebra if and only if
  $A$ is a Ditkin algebra.
\end{cor}

\subsection{Strong Ditkin's condition}
If $\A$ is a strong Ditkin algebra then $\A$ is strongly regular and $\mis(\A) = X\times \mis(A)$.
Also, by
Theorem \ref{thm:if-AA-has-P-then-fA-and-A-have-P}~(\ref{item:(st)Ditkin}),
the associated algebras $\fA$ and $A$ are strong Ditkin. We show that the converse is true
if each of the following conditions holds;
\begin{enumerate}
  \item $\A$ is an $A$-valued uniform algebra,
  \item $\fA A$, the subalgebra of $\A$ generated by $\fA\cup A$, is dense in $\A$.
\end{enumerate}

\begin{thm}\label{thm:st-Ditkin-for-uniform-algebras}
 Let $\A$ be an $A$-valued uniform algebra.
 If $\fA$ and $A$ are strong Ditkin algebras then $\A$ is a strong Ditkin algebra.
\end{thm}

\begin{proof}
  By Theorem \ref{thm:regularity-for-vv-FA}, the algebra $\A$ is regular
  and $\mis(\A)=X\times \mis(A)$. Take $x_0\in X$ and $\phi_0\in \mis(A)$,
  and let $\tau_0=\e_{x_0} \dm \phi_0$. We show that $\A$ satisfies strong Ditkin's condition
  at $\tau_0$ (equivalently, $I_{\tau_0}$ has a bounded approximate identity in $J_{\tau_0}$).
  First, we show that, for some $M>0$, for every $\e>0$, and for every finite set $\cF$ in
  $I_{\tau_0}$, there is an element $G=G_{\cF,\e}$ in $J_{\tau_0}$ such that $\|G\|_X\leq M$ and
    \begin{equation*}
       \|f-f G\|_X \leq \e \quad (f\in\cF).
    \end{equation*}

  Let $\cF=\set{f_1,\dotsc,f_n}$ be a finite set in $I_{\tau_0}$ and set
  $a_i=f_i(x_0)$, $1\leq i \leq n$. Then $\phi_0(a_i) = \phi_0\bigl(f_i(x_0)\bigr)=\tau_0(f_i)=0$.
  Since $A$ satisfies strong Ditkin's condition at $\phi_0$, the ideal $I_{\phi_0}$ has
  a bounded approximate identity $(u_\alpha)$ in $J_{\phi_0}$ with bound $M_{\phi_0}$, say.
  Choose $\alpha$ such that $\|a_i - a_i u_\alpha \| < \e$, for all $i=1,\dotsc,n$, and let $V_0$ be a neighbourhood
  of $\phi_0$ such that $\hat u_\alpha=0$ on $V_0$. Set
  \begin{equation*}
    U=\bigset{x\in X : \|f_i(x) - f_i(x)u_\alpha\| < \e, i=1,\dotsc,n}.
  \end{equation*}

  Then $U$ is a neighbourhood of $x_0$ in $X$ and $K=X\setminus U$ is compact in $X$.
  Since the uniform algebra $\fA$ satisfies strong Ditkin's
  condition at $x_0$, it admits bounded relative units at that point. Hence, there is a constant
  $M_{x_0}>0$, a neighbourhood $U_0$ of $x_0$, and a function $g\in \fA$ such that
  \begin{equation*}
     g|_{U_0}=0,\ g|_{X\setminus U}=1,\ \|g\|_X \leq M_{x_0}.
  \end{equation*}

  Let $M=M_{x_0}+ M_{\phi_0}+M_{\phi_0}M_{x_0}$. If $G_{\cF,\e}= g + u_\alpha - gu_\alpha$,
  then $\hat G_{\cF,\e}=0$ on $U_0 \times V_0$ so that $G_{\cF,\e}\in J_{\tau_0}$, and
  $\|G_{\cF,\e}\|_X \leq M$. Moreover, for every $x\in X$, we have
  \begin{align*}
    \|f_i(x) - f_i(x) G_{\cF,\e}(x) \|
      & = \| \bigl(\U-g(x)\bigr)\bigl(f_i(x)-f_i(x)u_\alpha\bigr)\| \\
      & \leq
      \left\{ \!\!\!
        \begin{array}{ll}
          0 , & \hbox{$x\in X\setminus U$;} \\
          (1+M_{x_0})\e, & \hbox{$x\in U$.}
        \end{array}
      \right.
  \end{align*}
  Therefore, $\|f_i-f_iG_{\cF,\e}\|_X \leq (1+M_{x_0})\e$, for all $f_i\in\cF$.

  Now, a traditional argument shows that
  the net $(G_{\cF,\e})$, where $\cF$ runs over all finite subsets of $I_{\tau_0}$ and $\e\in (0,\infty)$,
  forms a bounded approximate identity in $J_{\tau_0}$ for $I_{\tau_0}$.
  Therefore, $\A$ satisfies strong Ditkin's condition at $\tau_0$.
\end{proof}

\begin{thm}\label{thm:st-Ditkin-when-fA.A-is-dense}
 For the $A$-valued function algebra $\A$, suppose that $\fA A$ is dense in $\A$.
 If\/ $\fA$ and $A$ are strong Ditkin algebras, then $\A$ is a strong Ditkin algebra.
\end{thm}

\begin{proof}
  As in the proof of Theorem \ref{thm:st-Ditkin-for-uniform-algebras}, we need to show that
  $I_{\tau_0}$ has a bounded approximate identity in $J_{\tau_0}$, where $\tau_0=\e_{x_0} \dm \phi_0$
  belongs to $\mis(\A)=X\times \mis(A)$. Let $(u_\alpha)$ be a bounded approximate identity for $I_{\phi_0}$,
  contained in $J_{\phi_0}$, and let $(g_\beta)$ be a bounded approximate identity for $I_{x_0}$,
  contained in $J_{x_0}$. For $\lambda=(\alpha,\beta)$, set $G_\lambda = g_\beta + u_\alpha - g_\beta u_\alpha$.
  It is easily verified that $(G_\lambda)$ is a bounded net contained in $J_{\tau_0}$. We show that
  $FG_\lambda \to F$, for every $F\in I_{\tau_0}$. First, assume that $F\in \fA A$
  is of the form $F=f_1a_1+\dotsb+f_na_n$. Write
  \begin{equation*}
    F = \sum_{i=1}^n \bigl(f_i-f_i(x_0)\bigr)a_i + \sum_{i=1}^n f_i(x_0)a_i
      = \sum_{i=1}^n g_i a_i + a,
  \end{equation*}
  where $g_i=f_i-f_i(x_0)$ and $a = \sum_{i=1}^n f_i(x_0)a_i$.
  Then, $g_i\in I_{x_0}$ and
  \begin{equation*}
    \phi_0(a) = \phi_0\Bigprn{\sum_{i=1}^n f_i(x_0)a_i} = \sum_{i=1}^n f_i(x_0)\phi_0(a_i)
    = \tau_0(F) = 0.
  \end{equation*}

  Therefore, $au_\alpha \to a$ and $g_i g_\beta \to g_i$, for $i=1,\dotsc,n$.
  Let $M$ be a positive number such that
  $\|u_\alpha\|\leq M$, for all $\alpha$, $\|g_\beta\|\leq M$, for all $\beta$,
  and $\|a_i\|\leq M$, for $i=1,\dotsc,n$. Then
  \begin{align*}
    \|F- FG_\lambda\|
     & = \Bignorm{\sum_{i=1}^n g_ia_i + a - \sum_{i=1}^n g_ia_i G_\lambda -aG_\lambda} \\
     & \leq \sum_{i=1}^n\|g_i- g_iG_\lambda\|\|a_i\| + \|a-aG_\lambda\| \\
     & \leq M \sum_{i=1}^n\|g_i - g_i g_\beta - (g_i- g_ig_\beta) u_\alpha\|
         + \|a-au_\alpha - g_\beta(a-au_\alpha)\| \\
     & \leq M(1+M) \sum_{i=1}^n\|g_i - g_i g_\beta\| + (1+M)\|a-au_\alpha\|.
  \end{align*}

  This yields that $\|F- FG_\lambda\|\to 0$. Note that
  \begin{equation*}
    \|G_\lambda\| = \|g_\beta + u_\alpha - g_\beta u_\alpha\| \leq 2M+M^2.
  \end{equation*}

  Now, consider the general case of $F\in I_{\tau_0}$. Since $\fA A$ is dense in $\A$,
  for every $\e>0$, there exists $F_1\in \fA A$ such that $\|F-F_1\|<\e/2$. Let
  $F_0=F_1-\tau_0(F_1)\U$. Then $F_0\in \fA A$ and $\tau_0(F_0)=0$ so that, by the previous part,
  we get $\|F_0 - F_0G_\lambda\|\to0$. Moreover,
  \begin{equation*}
    \|F-F_0\| = \|F-F_1+\tau_0(F_1)\U\| \leq \|F-F_1\| + |\tau_0(F_1-F)| \leq \frac\e2+\frac\e2 = \e.
  \end{equation*}
  Hence, for every $\lambda$, we get
  \begin{align*}
   \|F - FG_\lambda\|
     & = \|(F-F_0) - (F-F_0)G_\lambda + F_0 - F_0 G_\lambda \|\\
     & \leq \|F-F_0\| (1+\|G_\lambda\|) + \|F_0 - F_0G_\lambda\| \\
     & \leq \e(1+2M+M^2) + \|F_0 - F_0G_\lambda\|.
  \end{align*}
  This yields that $\limsup_\lambda \|F - FG_\lambda\| \leq \e(1+M)^2$. Since $\e$ is arbitrary,
  we get
  \begin{equation*}
    \lim_\lambda \|F-FG_\lambda\| = 0 \quad (F\in I_{\tau_0}).
    \qedhere
  \end{equation*}
\end{proof}

As a consequence of the above theorem, we get the following result.

\begin{cor}\label{cor:strong-Ditkin-for-fA-tp-A}
 If\/ $\fA$ and $A$ are strong Ditkin algebras, then any Tomiyama product $\fA \htp_\gamma A$
 is a strong Ditkin algebra.
\end{cor}

\paragraph{Acknowledgment} The authors would like to express their sincere gratitude to the
anonymous referee for their careful reading and suggestions that improve the presentation
of the paper.


\begin{thebibliography}{99}

\bibitem{Abtahi-BJMA}
 M. Abtahi,
 \textit{Vector-valued characters on vector-valued function algebras},
 Banach J. Math. Anal. \textbf{10}(3) (2016), 608--620.

\bibitem{Abtahi-Farhangi-spectra}
 M. Abtahi and S. Farhangi,
 \textit{Vector-valued spectra of Banach algebra valued continuous functions},
 Rev. Real Acad. Cienc. Exac., Fis. Nat., Ser. A, Mat.,
 \textbf{112}(1) (2018), 103--115.

\bibitem{Abtahi-character-space}
 M. Abtahi,
 \textit{On the character space of Banach vector-valued function algebras},
 Bull. Iran. Math. Soc. \textbf{43}(5) (2017), 1195--1207.


\bibitem{Bade-Open-problem}
  W. G. Bade, \textit{Open problem},
  in Conference on Automatic Continuity and Banach Algebras,
  Proceeding of the Center for Mathematical Analysis \textbf{21}
  (Australian National university, 1989).

\bibitem{Bade-Dales-1992}
 W. G. Bade, H. G. Dales,
 \textit{The Wedderburn decomposability of some commutative Banach algebras},
 J. Funct. Anal. \textbf{107}(1) (1992), 105--121.


\bibitem{BD}
 F. F. Bonsall, J. Duncan,
 \textit{Complete Normed Algebras},
 Springer-Verlag, Berlin, Heidelberg, New York, (1973).


\bibitem{Dales}
 H. G. Dales,
 \textit{Banach Algebras and Automatic Continuity},
 London Mathematical Society Monographs, New Series 24, Oxford
 Science Publications, Clarendon Press, Oxford University Press,
 New York, 2000.

\bibitem{Fein-1992}
  J. F. Feinstein,
  \textit{Non-trivial, strongly regular uniform algebras},
  J. London Math. Soc. \textbf{2}(2) (1992), 288--300.

\bibitem{Fein-1995-note}
 J. F. Feinstein,
 \textit{A note on strong Ditkin algebras}, Bull. Aust. Math. Soc.
 \textbf{52}(1) (1995), 25--30.


\bibitem{Fein-1999}
 J. F. Feinstein,
 \textit{Strong Ditkin algebras without bounded relative units},
  Int. J. Math. Math. Sci. \textbf{22}(2) (1999), 437--443.

\bibitem{Fein-Somer-1999}
 J. F. Feinstein and D. W. B. Somerset,
 \textit{Strong regularity for uniform algebras},
 Contemp. Math. \textbf{232} (1999), 139--150.

\bibitem{Hausner}
  A. Hausner,
  \textit{Ideals in a certain Banach algebra},
  Proc. Amer. Math. Soc. \textbf{8}(2) (1957), 246--249.

\bibitem{CBA}
  E. Kaniuth,
  \textit{A Course in Commutative Banach Algebras},
  Graduate Texts in Mathematics, 246 Springer, 2009.

\bibitem{Mortini}
  R. Mortini,
  \textit{Closed and prime ideals in the algebra of bounded analytic functions},
  Bull. Aust. Math. Soc. \textbf{35} (1987), 213--229.


\bibitem{Nikou-Ofarrell-2014}
  A. Nikou and A. G. O'Farrell,
  \textit{Banach algebras of vector-valued functions},
  Glasgow Math. J. \textbf{56} (2014), no. 2, 419--246.

\bibitem{Nikou-Ofarrell-Ditkin}
   A. Nikou and A. G. O'Farrell,
   \textit{Ditkin conditions},
   Glasgow Math. J. \textbf{60}(1) (2018), 153--163.

\bibitem{Ryan}
  R. Ryan,
  Introduction to Tensor Products of Banach Spaces.
  Springer, London, 2002.


\bibitem{Tomiyama}
  J. Tomiyama,
  \textit{Tensor products of commutative Banach algebras},
  Tohoku Math. \textbf{12}(1) (1960), 147--154.

\bibitem{Willken-1969}
  D. R. Willken,
  \textit{A not of strongly regular uniform algebras},
  Canad. J. Math. \textbf{21} (1969), 912--914

\end{thebibliography}
\end{document}